\documentclass[12pt,a4paper,leqno,verbatim]{amsart}
\newcounter{minutes}
\divide\time by 60
\newcounter{hours}
\multiply\time by 60 \addtocounter{minutes}{-\time}

\usepackage{amssymb}
\usepackage{hyperref}
\usepackage{graphicx}
\date{}
\newfont{\cyrilic}{wncyr10 scaled 1000}
\title{New trigonometric and hyperbolic inequalities}
\author[B.A Bhayo]{Barkat Ali Bhayo}
\address{Koulutuskeskus Salpaus (Salpaus Further Education), 15110 Lahti, Finland}
\email{bhayo.barkat@gmail.com}

\author[R. Kl\'en]{Riku Kl\'en}
\address{Department of Mathematics and Statistics, University of Turku, FI-20014 Finland}
\email{riku.klen@utu.fi}

\author[J. S\'andor]{J\'ozsef S\'andor}
\address{Babe\c{s}-Bolyai University
Department of Mathematics
Str. Kogalniceanu nr. 1
400084 Cluj-Napoca, Romania}
\email{jsandor@math.ubbcluj.ro}

\newcommand{\comment}[1]{}

\swapnumbers
\theoremstyle{plain}

\newtheorem{theorem}[equation]{Theorem}
\newtheorem{lemma}[equation]{Lemma}

\newtheorem{corollary}[equation]{Corollary}

\theoremstyle{definition}
\newtheorem{remark}[equation]{Remark}

\numberwithin{equation}{section}

\begin{document}

\def\thefootnote{}
\footnotetext{ \texttt{\tiny File:~\jobname .tex,
          printed: \number\year-\number\month-\number\day,
          \thehours.\ifnum\theminutes<10{0}\fi\theminutes}
} \makeatletter\def\thefootnote{\@arabic\c@footnote}\makeatother

\maketitle

\begin{abstract}
The aim of this paper is to prove new trigonometric and
hyperbolic inequalities , which constitute among others refinements or
analogs of famous Cusa-Huygens, Wu-Srivastava, and related inequalities. In
most cases, the obtained results are sharp.
\end{abstract}




\section{introduction}

Since the last decade many authors have been interested in finding upper and lower bounds for expression $f(x)/x$, where $f(x)$ is a trigonometric or a hyperbolic function. We continue this line of research.

The inequality
\begin{equation}\label{laztri}
(\cos x)^{1/3}<\frac{\sin x}{x}<\frac{\cos x+2}{3}
\end{equation}
holds for $0<|x|<\pi/2$. The left hand side inequality is due to by D.D. Adamovi\'c and D.S. Mitrinovi´c \cite[p. 238]{mit}, and the right hand side one was obtained by N. Cusa and C. Huygens in 2005 \cite{sanben}. The hyperbolic version of the above inequality is given as 
\begin{equation}\label{lazhyp}
(\cosh x)^{1/3}<\frac{\sinh x}{x}<\frac{\cosh x+2}{3},
\end{equation}
for $x\neq 0$. The left hand side inequality in \eqref{lazhyp} was obtained by Lazarevi\'c \cite[p. 270]{mit}, 
and the right hand side inequality is called sometime
hyperbolic Cusa-Huygens inequality \cite{neusan}.

In 1989 J. Wilker \cite{wi} discovered the following inequality
\begin{equation}\label{wilka}
\left(\frac{\sin x}{x}\right)^2+\frac{\tan x}{x}>2,\quad 0<|x|<\frac{\pi}{2}.
\end{equation}
Numerous authors have studied this inequality by giving simpler proofs and generalization, e.g, see
\cite{guo,pi,neusan,sum,ws1,ws2,z1,z2}. The following inequality is due to Huygen \cite{hyu}
\begin{equation}\label{huyineq}
2\frac{\sin x}{x}+\frac{\tan x}{x}>3,\quad 0<|x|<\frac{\pi}{2}
\end{equation}
and is called Huygens inequality or the first Wilker inequality.
In \cite{ws1}, Wu and Srivastava introduced the following inequality
\begin{equation}\label{wilk}
\left(\frac{x}{\sin x}\right)^2+\frac{x}{\tan x}>2,\quad 0<|x|<\frac{\pi}{2}.
\end{equation}
which is called the second Wilker inequality.

We study inequalities \eqref{laztri}--\eqref{wilk} and our main results are the following five theorems. Our first main result is a generalization of Cusa-Huygens inequality \eqref{laztri}.

\begin{theorem}\label{thm1} For $x\in [-\pi/2,\pi/2]$, we have
\begin{equation}\label{thm1ineq}
\frac{\cos x+\alpha-1}{\alpha} \leq \frac{\sin x}{x} \leq \frac{\cos x+\beta-1}{\beta},
\end{equation}
with the best possible constants $\alpha=\pi/(\pi-2)\approx 2.75194$ and $\beta=3$. The lower bound is sharp for $x \in \{ -\pi/2 , 0 ,\pi/2 \}$ and the upper bound is sharp for $x=0$.
\end{theorem}

\begin{remark}
 The upper bound of Theorem \ref{thm1} is sharp at point $x=0$ and the lower bound at points $x=-\pi/2$, $x=0$ and $x=\pi/2$. For values $x\in [-\pi/2,\pi/2]$ the difference between the function and the lower bound is less than 0.01 and the difference between the function and the upper bound less than 0.031. The right hand side inequality holds true for all real numbers $x$ and as a sharp inequality for all $x \neq 0$.
\end{remark}

In the following two theorems we introduce Wu-Srivastava type inequalities for the trigonometric functions.

\begin{theorem}\label{thm0}
For $x\in [-\pi/2,\pi/2]$, we have
\begin{equation}\label{thm0ineq}
\left(\frac{x}{\sin x}\right)^2+\left(\frac{\pi^2}{4}-1\right)\frac{x}{\tan x} \leq \frac{\pi^2}{4}
\end{equation}
and the equality is attained at values $x=-\pi/2$, $x=0$ and $x=\pi/2$.
\end{theorem}

\begin{remark}
  The upper bound of Theorem \ref{thm0} is sharp at points $x=-\pi/2$, $x=0$ and $x=\pi/2$. For values $x\in [-\pi/2,\pi/2]$ the difference between the function and the upper bound is less that 0.13.
\end{remark}

\begin{theorem}\label{newthm} For $x\in[-\pi/2,\pi/2]$, we have
\begin{equation}\label{newineq1}
(\alpha-1)\frac{x}{\sin x}+\frac{x}{\tan x} \leq \alpha,
\end{equation}
\begin{equation}\label{newineq2}
\left(\frac{x}{\sin x}\right)^\alpha+\frac{x}{\tan x} < \left(\frac{\pi}{2}\right)^\alpha,
\end{equation}
with the best possible constant $\alpha=\pi/(\pi-2)$. The inequality \eqref{newineq1} is sharp for $x \in \{ -\pi/2,0,\pi/2 \}$ and the inequality \eqref{newineq2} is sharp for $x \in \{ -\pi/2,\pi/2 \}$.
\end{theorem}

\begin{remark}
  For values $x\in [-\pi/2,\pi/2]$ the difference between the function $\alpha$ and the lower bound is less than 0.031. In inequality \eqref{newineq2} the difference between the function and the lower bound is between 1.45 and 1.9.
\end{remark}

Next theorem is a generalization of the Huygens inequality.

\begin{theorem}\label{thm2} For $x\in(-\pi/2,\pi/2)$, we have
\begin{equation}\label{thm2ineq}
3\cos x \le \frac{x}{\sin x}+2\frac{x}{\tan x} \le 2+\cos x.
\end{equation}
The inequalities are sharp at $x=0$.
\end{theorem}

The following result is a Wilker and Wu-Srivastava type result for the hyperbolic functions.

\begin{theorem}\label{thm4} For $x>0$, we have
\begin{enumerate}
\item $\displaystyle\left(\frac{x}{\sinh x}\right)^2+\frac{x}{\tanh x}<\left(\frac{\sinh x}{x}\right)^2+
\frac{\tanh x}{x}<\frac{1+\cosh (2x/3)}{2}\left(\left(\frac{x}{\sinh x}\right)^2+\frac{x}{\tanh x}\right).$\\
\end{enumerate}
\end{theorem}

Our last result is a counterpart of
the inequality 
\begin{equation}\label{yanginequ}
\exp(-x^2/6)<\frac{2+\cos x}{3},\quad 
\end{equation}
$x\in(0,\infty)$, which recently appeared in \cite[Thm 2]{yang}. 

\begin{theorem}\label{2702} For $x\in(0,\pi/2)$, the following inequalities hold
$$\exp \left( \alpha-\frac{(\pi-2)x^2}{2\pi} \right) <\frac{(\pi-2)\cos(x)+2}{\pi}<\exp \left( \beta-\frac{(\pi-2)x^2}{2\pi} \right) ,$$
with the best possible constants $\alpha=(\pi^2+8\log(2/\pi)-2\pi)/8\approx -0.00328$ and $\beta=0$.
\end{theorem}

\section{Preliminaries and lemmas}

\begin{lemma}\label{lembk}
For $0<R\leq \infty$. Let $A(x)=\sum_{n=0}^\infty a_nx^n$ and 
$C(x)=\sum_{n=0}^\infty c_nx^n$ be two real power series converging on the interval $(-R,R)$. If the sequence
$\{a_n/c_n\}$ is increasing (decreasing) and $c_n>0$ for all $n$, then the function $A(x)/C(x)$ is also
increasing (decreasing) on $(0,R)$.
\end{lemma}

For $|x|<\pi$, the following power series expansions can be found in \cite[1.3.1.4 (2)--(3)]{jef},
\begin{equation}\label{xcot}
x \cot x=1-\sum_{n=1}^\infty\frac{2^{2n}}{(2n)!}|B_{2n}|x^{2n},
\end{equation}

\begin{equation}\label{cot}
\cot x=\frac{1}{x}-\sum_{n=1}^\infty\frac{2^{2n}}{(2n)!}|B_{2n}|x^{2n-1},
\end{equation}
and 
\begin{equation}\label{coth}
{\rm \coth} \, x=\frac{1}{x}+\sum_{n=1}^\infty\frac{2^{2n}}{(2n)!}|B_{2n}|x^{2n-1},
\end{equation}
where $B_{2n}$ are the even-indexed Bernoulli numbers, see \cite[p. 231]{IR}. 
We can get the following expansions directly from (\ref{cot}) and (\ref{coth}),

\begin{equation}\label{cosec}
\frac{1}{(\sin x)^2}=-(\cot x)'=\frac{1}{x^2}+\sum_{n=1}^\infty\frac{2^{2n}}{(2n)!}
|B_{2n}|(2n-1)x^{2n-2},
\end{equation}

\begin{equation}\label{cosech}
\frac{1}{(\sinh x)^2}=-({\rm coth} \, x)'=\frac{1}{x^2}-\sum_{n=1}^\infty\frac{2^{2n}}{(2n)!}(2n-1)|B_{2n}|x^{2n-2}.
\end{equation}
For the following expansion formula 
\begin{equation}\label{xsin}
\frac{x}{\sin x}=1+\sum_{n=1}^\infty\frac{2^{2n}-2}{(2n)!}|B_{2n}|x^{2n}
\end{equation}
see \cite{li}.


\begin{lemma}\cite[Theorem 2]{avv1}\label{lem0}
For $-\infty<a<b<\infty$,
let $f,g:[a,b]\to \mathbb{R}$
be continuous on $[a,b]$, and differentiable on
$(a,b)$. Let $g^{'}(x)\neq 0$
on $(a,b)$. If $f^{'}(x)/g^{'}(x)$ is increasing
(decreasing) on $(a,b)$, then so are
$$\frac{f(x)-f(a)}{g(x)-g(a)}\quad and \quad \frac{f(x)-f(b)}{g(x)-g(b)}.$$
If $f^{'}(x)/g^{'}(x)$ is strictly monotone,
then the monotonicity in the conclusion
is also strict.
\end{lemma}

\begin{lemma}\label{lema} The following function
$$f_1(x)=\frac{(x/\sin x)^2-x\cot x}{1-x\cot x}$$
is strictly increasing from $(0,\pi/2)$ onto $(\pi^2/4)$.
\end{lemma}

\begin{proof}
Let $f_1(x)=A_1(x)/C_1(x)$, where 
$$A_1(x)=(x/\sin x)^2-x\cot x \quad {\rm and}\quad C_1(x)=1-x\cot x.$$
By using expansion formulas \eqref{cot} and \eqref{cosec} we get
\begin{eqnarray*}
A_1(x)&=&1+\sum_{n=1}^\infty\frac{2^{2n}}{(2n)!}(2n-1)|B_{2n}|x^{2n}-1
+\sum_{n=1}^\infty\frac{2^{2n}}{(2n)!}|B_{2n}|x^{2n}\\
&=&\sum_{n=1}^\infty\frac{2^{2n}2n}{(2n)!}|B_{2n}|x^{2n}=\sum_{n=1}^\infty a_nx^{2n},
\end{eqnarray*}
and
$$C_1(x)=\sum_{n=1}^\infty\frac{2^{2n}}{(2n)!}|B_{2n}|x^{2n}=\sum_{n=1}^\infty c_nx^{2n}.$$
Let $d_n=a_n/c_n=2n$, which is the increasing in $n\in\mathbb{N}$. Thus, by Lemma \ref{lembk} $f_1(x)$
is strictly increasing in $x\in(0,\pi/2)$. Applying l'H\^opital rule, we get $\lim_{x\to 0}f_1(x)=2$ 
and $\lim_{x\to \pi/2}f_1(x)=\pi^2/4$. This completes the proof.
\end{proof}

\begin{lemma}\label{lem2b} One has
\begin{equation}\label{lem2bineq1}
\frac{\tanh x}{x} \le \frac{2}{\sqrt{9+4x^2}-1},\quad x \in \mathbb{R},
\end{equation}
\begin{equation}\label{ineq0515}
\frac{\sinh x}{x}<\frac{\cosh x+2}{3}<(\cosh x)^{1/3}\frac{\cosh (2x/3)+1}{2},\quad x>0.
\end{equation}
\end{lemma}

\begin{proof} Clearly both sides of the inequality \eqref{lem2bineq1} get value 1 at $x=0$. By symmetry of the function we need consider only the positive values of $x$. Let 
$$f_2(x)=\left( \frac{2x}{\tanh x}+1\right)^2-4x^2-9.$$
We get, 
\begin{eqnarray*}
f_2'(x) &=&4\left({\rm coth} x+\frac{x(1-2x \,{\rm coth} \, x)}{(\sinh x)^2}\right)\\
&=& 4\frac{x^2\cosh x}{(\sinh x)^3}\left(\left(\frac{\sinh x}{x}\right)^2+\frac{\tanh x}{x}-2\right)\\
&=& 4\frac{x^2\cosh x}{(\sinh x)^3}f_3(x)>0,
\end{eqnarray*}
where the inequality follows since $f_3(x)>0$ is equivalent to
\[
  \frac{\sinh x}{x}+\frac{1}{\cosh x}>\frac{x}{\sinh x},
\]
which is clearly true because $(\sinh x)/x > 1$. Now $f_2$ is strictly increasing, and $\lim_{x\to 0}f_2(x)=0<f_2(x)$. This implies the proof of 
\eqref{lem2bineq1}.

The first inequality in \eqref{ineq0515} is well known, for the second inequality we define
$$f_4(x)=\frac{\cosh x+2}{3}-(\cosh x)^{1/3}\frac{\cosh (2x/3)+1}{2}.$$
Simple computation gives
%
$$f'_4(x)=\frac{\sinh \frac{x}{3}-2\sinh x+4 \cosh^{2/3}x \sinh x-3 \sinh \frac{5x}{3}}{12(\cosh x)^{2/3}}$$
and clearly $\sinh \frac{x}{3}-2\sinh x < 0$. We prove that $4 \cosh^{2/3}x \sinh x-3 \sinh \frac{5x}{3} < 0$, which is equivalent to
\[
  f_5(x) = 4 \sinh x+2\sinh (3x)+ \frac{11}{8} \sinh (5x)-\frac{81}{8}\sinh \frac{5x}{3} > 0.
\]
Since
\begin{eqnarray*}
  f_5''(x) & = & 4\sinh x+18\sinh (3x)+\frac{275}{8}\sinh (5x)-\frac{225}{8}\sinh \frac{5x}{3}\\
  & \ge & \frac{275}{8}\sinh (5x)-\frac{225}{8}\sinh \frac{5x}{3} > 0
\end{eqnarray*}
for all $x>0$, it is clear that $f_5'(x)>f_5'(0)=0$ and $f_5(x)$ is increasing. Thus $f_5(x)>f_5(0)=0$.
\end{proof}

\begin{remark}
  In Lemma \eqref{lem2bineq1} the difference between the function and the upper bound is less than 0.02. The upper bound is asymptotically sharp $\frac{2}{\sqrt{9+4x^2}-1} - \frac{\tanh x}{x} \to 0$ as $x \to \pm \infty$.
\end{remark}

\section{Proofs of the main results}
In this section we give the proof of our theorems.
\vspace{.3cm}

\noindent{\bf Proof of Theorem \ref{thm1}.}
Sharpness of the bounds is obvious. Since the bounds and the function $\frac{\sin x}{x}$ are even we need to prove sharp inequality for $x \in (0,\pi/2)$. Let $$f_6(x)=\frac{\cos x-1}{(\sin x)/x-1}=\frac{(\sin x)/x-x\cot x}{x/\sin x-1}=\frac{A_2(x)}{C_2(x)}.$$
Using series expansion formulas \eqref{xcot} and \eqref{xsin}, we get

\begin{eqnarray*}
A_2(x)&=&\sum_{n=1}^\infty\frac{2^{2n}-2}{(2n)!}|B_{2n}|x^{2n}
+\sum_{n=1}^\infty\frac{2^{2n}}{(2n)!}|B_{2n}|x^{2n}\\
&=&\sum_{n=1}^\infty\frac{2(2^{2n}-1)}{(2n)!}|B_{2n}|x^{2n}=\sum_{n=1}^\infty \tilde{a}_nx^{2n},
\end{eqnarray*}
and
$$C_2(x)=\sum_{n=1}^\infty\frac{2^{2n}-2}{(2n)!}|B_{2n}|x^{2n}=\sum_{n=1}^\infty \tilde{c}_nx^{2n}.$$
Write $\tilde{d}=\tilde{a}/\tilde{c}=2(2^{2n}-1)/(2^{2n}-2)$. Clearly $\tilde{d}$ is a decreasing 
function of $n\in\mathbb{N}$. Hence $f_6$ is decreasing by Lemma \ref{lembk}. 
Applying l'H\^opital rule, we get $\lim_{x\to 0}f_6(x)=3$ 
and $\lim_{x\to \pi/2}f_6(x)=\pi/(\pi-2)$, this finishes the proof.
$\hfill\square$

\vspace{.3cm}

\noindent{\bf Proof of Theorem \ref{thm0}.} Equality in the claim is clearly attained at points $x=-\pi/2$, $x=0$ and $x=\pi/2$. Since the left hand side of the inequality is an even function we need to show that the inequality is sharp for $x \in (0,\pi/2)$. Let 
$$\frac{x/\sin x-\cos x}{(\sin x)/x-x\cot x}=\frac{(x/\sin x)^2-x\cot x}{1-x\cot x}=f_1(x).$$
By Lemma \ref{lema}, we get
$$2<\frac{x/\sin x-\cos x}{(\sin x)/x-x\cot x}<\frac{\pi^2}{4},$$
this implies the following inequalities
\begin{equation}\label{sinxnew}
\frac{4}{\pi^2}\left(\frac{x}{\sin x}+\left(\frac{\pi^2}{4}-1\right)\cos x\right)<\frac{\sin x}{x}
<\frac{1}{2}\left(\frac{x}{\sin x}+\cos x\right).
\end{equation}
The first inequality of \eqref{sinxnew} can be written as \eqref{thm0ineq}. This completes the proof.$\hfill\square$
\vspace{.3cm}

The second inequality of \eqref{sinxnew} is also proved by Neuman and S\'andor, see \cite[Theorem 2.3]{neusan},
They pointed out that it can be written as the the Wu-Srivastava inequality.

\begin{lemma}\label{lem1202}
Let $\alpha = \pi/(\pi-2)$ as in Theorem \ref{thm1}. The function 
$$f_\alpha(b)=\left(\frac{\alpha}{\alpha+b-1}\right)^\alpha+\frac{\alpha b}{\alpha+b-1}$$
is decreasing from $(0,1)$ onto $(2,k)$, where $k=(\pi/2)^\alpha \approx 3.46505$.
\end{lemma}

\begin{proof} We get
\begin{eqnarray*}
f'_\alpha(b)&=&\frac{a \left(a-1-(a+b-1)
   \left(a/(a+b-1)\right)^a\right)}{(a+b-1)^2}\\
	&=& -\frac{\pi(\pi ^{\pi/(\pi-2)} (2+\pi(\pi
   -2))^{(2-\pi)/2}-2)}{(2+b(\pi -2))^2},
	\end{eqnarray*}
	which is negative, and $f_\alpha$ tends to $k$ and 2 when $x$ tends to $0$ and $1$, respectively.
\end{proof}

\noindent{\bf Proof of Theorem \ref{newthm}.} The inequality \eqref{newineq1} follows from the first inequality of
\eqref{thm1ineq}. Again, utilizing the same inequality
$$\frac{x}{\sin x}< \frac{\alpha}{\alpha -1 +\cos x},$$
we get
$$\left(\frac{x}{\sin x}\right)^\alpha+\frac{x}{\tan x}<\left(\frac{\alpha}{\alpha-1+\cos x}\right)^\alpha+
\frac{\alpha \cos x}{\alpha-1+\cos x}=f_\alpha(\cos x)<\left(\frac{\pi}{2}\right)^\alpha$$
by Lemma \ref{lem1202}, because $f_\alpha(\cos x)$ is strictly increasing and $\lim_{x\to 1}f_\alpha(\cos x)=(\pi/2)^\alpha$.
$\hfill\square$

\vspace{.3cm}


\noindent{\bf An other proof of \eqref{newineq1}.}
Let
$$f(x)=  (\alpha-1)\frac{x}{\sin x}  + \frac{x}{\tan x} $$
for $x\in(0,\pi/2)$. 
An easy computation gives    
$$(\sin x)^2\cdot f'(x)=  (\alpha-1)\sin x-(\alpha-1)x\cos x  
+ \sin x\cos x -x=   g(x).$$
One has  $g'(x)=  (2(\sin x)/x)\cdot h(x)$, where  $h(x)= (\alpha-1)/2 - (\sin x)/x$.
As the function of $x$, $(\sin x)/x$ is strictly decreasing, the equation 
$(\alpha-1)/2= (\sin x)/x$ has at most a single root in $(0, \pi/2)$.  

Suppose that $(\alpha-1)/2 <1$, then the equation has exactly one root  $x_0\approx 0.8795$. 
As $h(x)<0$ for $x<x_0$ and $h(x) >0$  for $x>x_0$, the function $g(x)$ will be strictly decreasing for $x$ in $(0, x_0)$, and strictly increasing in $(x_0, \pi/2)$. 
 
Suppose that  $\alpha>\pi/2 +1$, which is equivalent to $(\alpha-1)/2 > \pi/4$.  Then we get  $g(\pi/2)= \alpha-1-\pi/2>0$. As $g(0)=0$, clearly $g(x_0)<0$, and $g$ will have a single root $x_1$ between $(x_0, \pi/2)$. Then we get 
 $g(x)<0$ for $x$ in $(0,x_1)$, and $g(x)>0$ for $x$ in $(x_1, \pi/2)$, 
where $x_1\approx 1.1559$.
This means that at the points $0$ and $\pi/2$ the function $f(x)$ will take the maximum values. Supposing $f(\pi/2)\leq f(0)$, the inequality  $f(x)\leq \alpha$  will be true.
Now, $f(\pi/2)\leq f(0)$ means exactly that $\alpha\leq   \pi/(\pi-2)$. Remark that  for  the best possible $\alpha=\pi/(\pi-2)$ one has also $\alpha >\pi/2 +1$  and 
$(\alpha-1)/2<1$, so the assumed properties in the proof are valid. This finishes the proof of inequality $f(x)\leq \alpha$, with best possible $\alpha=\pi/(\pi-2)$.$\hfill\square$

\begin{remark}
Indeed, the point $x_1$ above is the minimum point of $f(x)$ on $(0, \pi/2)$. In fact the following converse of inequality \eqref{newineq1} 
\begin{equation}\label{indep}
(\alpha-1)\frac{x}{\sin x}  + \frac{x}{\tan x}\geq f(x_1)\approx 2.7219
\end{equation}
holds true for  $\alpha= \pi/(\pi-2), \,(\alpha-1)/2<1$  and $\alpha>\pi/2 +1$ 
(as $\alpha <3$, it is sufficient to suppose   $\pi/2+1<\alpha \leq \pi/(\pi-2)$).
\end{remark}

\noindent{\bf Proof of Theorem \ref{thm2}.}
Clearly the function, its lower and upper bound get value 3 at origin. By symmetry of the function we consider only values $x \in (0,\pi/2)$. The second inequality in \eqref{thm2ineq} is equivalent to write
$$f_7(x)=\left(\frac{2+\cos x}{3\cos x}\right)-\left(\frac{2x}{\sin x}+\frac{x}{\sin x\cos x}\right)>0.$$
It is sufficient to prove that $f_7>0$. Using the the following inequalities
$$(\cos(x/2))^{4/3}<\frac{\sin x}{x}<\frac{2+\cos x}{3}$$ we get,
\begin{eqnarray*}
f_7(x)&>&\frac{3}{\cos x}\frac{\sin x}{x}-\frac{x}{\sin x}\left(2+\frac{1}{\cos x}\right)\\
&=&\frac{3x}{\sin x\cos x}\left(\left(\frac{\sin x}{x}\right)^2-\frac{2\cos x+1}{3}\right)\\
&>&\frac{3x}{\sin x\cos x}\left(\left(\frac{1+\cos x}{2}\right)^{4/3}-\frac{2\cos x+1}{3}\right)\\
&=&\frac{6x}{\sin 2x}f_8(x).\\
\end{eqnarray*}
For showing that $f_8$ is positive, we define the function
$$f_9(y)=\frac{(1+y)^4}{(2y+1)^3},\quad y\in(0,1),$$
and get 
$$f'_9(y)=\frac{(1-y)(1+y)^3}{(1+2y)^4}<0,$$
with $f_9(1)=16/27$. This implies that $f_8$ is positive, and this completes the proof of the second inequality.
$\hfill\square$

\vspace{.3cm}

\begin{remark}
  The upper bound of Theorem \ref{thm2} holds true for values $x \in (-\pi,\pi)$. The difference between the function and the lower bound is less than 1.6 and between the function and the upper bound is less than 0.55. In both cases the difference is less than $x^2$.
\end{remark}

\begin{corollary}\label{jozs} For $x\in(0,\pi/2)$, we have
\begin{equation}\label{joz2811a}
\frac{\pi}{2}   + \cos x<  \frac{x}{\sin x} +2\frac{x}{\tan x} <  2 + \cos x, 
\end{equation}
\end{corollary}

\begin{proof} Let $f(x)=  x/\sin x + 2x/\tan x –\cos x$ for $x\in(0,\pi/2)$.
After elementary computations, we get   

$$f'(x)\cdot (\sin x)^2 =  \sin x-x\cos x +2\sin x \cos x -2x +(\sin x)^3  =h(x).$$   One has 
$h'(x)=  \sin x\cdot k(x),$  where   
$$k(x)=  x-4\sin x +3\sin x\cos x.$$  
As   
$k'(x)= 4 (\cos x )^2- 4\cos x-2 (\sin x)^2,$  by    
$0<\cos x<1$ we clearly get $k'(x)<0$. Thus $k(x)<k(0)=0$, implying $h(x)<h(0)=0$. Finally, we get $f'(x)<0$. Thus $f(x)$ is strictly decreasing, this implies
$f(\pi/2)<f(x)< f(0)$, so the result follows.
\end{proof}

We remark here that the first inequality of \eqref{joz2811a} cannot be compared with the first inequality of \eqref{thm2ineq}, and obviously the right hand sides of both inequalities are the same.

\begin{corollary} The inequality \eqref{lem2bineq1} implies the hyperbolic version of the Wu-Srivastava inequality \eqref{wilk}.
\end{corollary}

\begin{proof} It is easy to see that the inequality \eqref{lem2bineq1} implies the validity of 
$$\left(\frac{ x}{\tanh x}\right)^2+\frac{ x}{\tanh x}-x^2-2>0.$$
The above inequality can be written as
$$\left(\frac{ x}{\sinh x}\right)^2+\frac{ x}{\tanh x}>0,$$
by observing that
\[
\left(\frac{ x}{\sinh x}\right)^2-x^2=\frac{ x}{\sinh x}>0.\qedhere
\]
\end{proof}

%
%

\vspace{.3cm}

\noindent{\bf Proof of Theorem \ref{thm4}.} The first inequality is well known and follows from \eqref{lazhyp}. Similarly the second inequality
follows from \eqref{ineq0515}.
$\hfill\square$

\begin{theorem}\label{thm2201} For $x\in(-\pi/2,\pi/2)$, one has
\begin{enumerate}
\item $\displaystyle\frac{\cos x+2}{3^{\alpha_1}}<\frac{\sin x}{x}<\frac{\cos x+2}{3^{\beta_1}},$\\
\item $\displaystyle\frac{\cos x+2^{\alpha_2}}{3}<\frac{\sin x}{x}<\frac{\cos x+2^{\beta_2}}{3}.$
\end{enumerate}
with best possible constants $\alpha_1=\log(\pi)/\log (3)\approx 1.04198,\,\alpha_2=\log(\pi/6)/\log (2)
\approx 0.93345$, $\beta_1=1$ and $\beta_2=1$.
\end{theorem}

\begin{proof} For (1), let $f_{10}(x)=\frac{x(2+\cos x)}{\sin x}$, with  $x \in (0, \pi/2)$.    
One has  
$$(\sin x)^2\cdot f'_{10}(x)=  2\sin x -2x\cos x +\sin x \cos x-x =  f_{11}(x).$$ 
As  $f'_{11}(x)=  2(\sin x) (x-\sin x )>0$, we get   $f_{11}(x)>f_{11}(0)= 0$,  so 
$f'_{10}(x)>0$, proving that $f_{10}(x)$ is strictly increasing. This implies that the function
$$f_{12}(x)=\frac{1}{\log 3}\log\left(\frac{x(\cos x+2)}{\sin x}\right)$$
is strictly increasing in $x\in(0,\pi/2)$, and we get
$$\beta_1=f_{12}(0+)<f_{12}(x)<f_{12}(\pi/2)=\alpha_1,$$
thus (1) follows.

For the proof of (2), write $f_{13}=3 (\sin x)/x -\cos x  ,\quad  x \in (0, \pi/2)$.
One has  $$x^2f'_{13}(x)=  3x\cos x + x^2 \sin x -3\sin x =  f_{14}(x).$$
Here  $f'_{14}(x)=  x(x\cos x -\sin x)<0$,  since  $x\cos x< \sin x$. This proves 
$f_{14}(x)<f_{14}(0)=0$, i.e.  $f'_{13}(x)<0$, proving that  $f_{13}(x)$ is strictly decreasing. This implies $f_{13}(x)<f_{13}(0+)= 2$  and $f_{13}(x)>f_{13}(\pi/2)= 6/\pi$ . 
As  $2=2^{\beta_2}$  for $\beta_2=1$, and $6/\pi= 2^{\alpha_2}$ for
$\alpha_2= \log(6/\pi))/\log (2)$, the result follows.
\end{proof}



\noindent{\bf Proof of Theorem \ref{2702}.} Let 
$$f(x)=\log\left(\frac{(\pi-2)\cos x+2}{\pi}\right)+\frac{(\pi-2)x^2}{2\pi}.$$
Simple calculation yields
$$f'(x)=(\pi-2)\left(\frac{x}{\pi}-\frac{\sin x}{2+(\pi-2)\cos x}\right),$$
which is negative by Theorem \ref{thm1}. Thus, the $f$ function is strictly decreasing, and
$$\alpha=\lim_{x\to \pi/2}f(x)<f(x)<0=\beta=\lim_{x\to 0}f(x).$$
This implies the proof.
$\hfill\square$

We finish the paper by giving a new type of Kober's inequality \cite{mit,sandor1801}, which follows easily from \eqref{yanginequ}
and Theorem \ref{2702}.
\begin{corollary}
For $x\in(0,\pi/2)$, the following inequalities hold
$$3\exp \left( -\frac{x^2}{6} \right) -2<\cos x<\frac{\pi\exp(-(\pi-2)x^2/(2\pi))-2}{\pi-2}.$$
\end{corollary}


\end{document}